\documentclass[a4paper,12pt]{amsart}

\usepackage[english]{babel}
\usepackage[autostyle,german=quotes]{csquotes}

\usepackage{amssymb}
\usepackage{amsthm}
\usepackage{amsmath}

\usepackage{enumitem}
\usepackage[pagewise,mathlines]{lineno}

\newtheorem{thm}{Theorem}

\newtheorem{lem}[thm]{Lemma}
\newtheorem{prop}[thm]{Proposition}

\theoremstyle{remark}
\newtheorem{rem}[thm]{Remark}
\newtheorem{example}[thm]{Example}

\newcommand{\Z}{{\mathbb{Z}}}
\newcommand{\Q}{{\mathbb{Q}}}
\newcommand{\G}{G}

\usepackage{hyperref}

\begin{document}
	
	\title[Hasse-Minkowski theorem for quadratic forms on groups]{Hasse-Minkowski theorem for quadratic forms on groups}
	
	\author{Stefan Bara\'{n}czuk} 
	\address{Faculty of Mathematics and Computer Science, Adam Mickiewicz University\\ul. Uniwersytetu Pozna\'{n}skiego 4, Pozna\'{n}, Poland}
	\email{stefbar@amu.edu.pl}
	
	\begin{abstract} 
		Consider groups such as  Mordell-Weil groups of abelian varieties over number fields, odd algebraic $K$-theory groups of number fields, or finitely generated subgroups of the multiplicative groups of number fields. They are all equipped with systems of reduction maps; thus, one can investigate the Hasse-Minkowski theorem for quadratic forms with coefficients in such groups. In this paper, we prove that the theorem holds for the forms whose rank equals $2$ or $3$, and we demonstrate that it does not hold for higher ranks by providing a counterexample. We also show that our results constitute a generalization of the classic Hasse-Minkowski theorem for binary and ternary integral forms.
	\end{abstract}

	

	\maketitle

	The Hasse-Minkowski theorem (cf. \cite{H}) states that a quadratic form over $\mathbb{Q}$ represents $0$ if and only if it represents $0$ in $\mathbb{Q}_{v}$ for every place $v$ on $\mathbb{Q}$, including $\infty$. For quadratic forms of the rank equal to $2$ or $3$, the following stronger result is known.
	
	\begin{lem}\hspace{1cm}\label{2and3}
		\begin{enumerate}[label=\emph{(\alph*)}]
			\item \label{P1}
			An integral quadratic form 
			\[a x^2 + b y^2\]
			represents $0$ if and only if it represents $0$ modulo all but finitely many prime numbers.  
			\item \label{P2}
			An integral quadratic form
			\[a x^2 + b y^2 + c z^2\]
			represents $0$ if and only if it represents $0$ modulo every $l^{k}$, where $l$ is a prime number and $k$ is a nonnegative integer.
		\end{enumerate}
	\end{lem}
	Since quadratic forms can be defined on any module over a commutative ring, it is a natural question whether we can generalize Lemma \ref{2and3} to modules equipped with an appropriate system of reduction maps. In this paper, we establish such a generalization to groups satisfying the following simple abstract axiomatic setup.
	
	Let $\G$ be an abelian group whose torsion subgroup $\G_{\mathrm{tors}}$ has finite exponent. Let $r_{v} \colon \G \to \G_{v}$ be a family of group homomorphisms indexed by a set of all but finitely many primes $v$ in a number field $\mathbb{K}$. The targets $\G_{v}$ are finite abelian groups. We  use the following notation:\\
	
	\begin{tabular}{ll}
		$P \bmod v$	&	denotes $r_{v}(P)$ for $P \in \G$\\
		$P=Q \bmod v$ & means  $r_{v}(P)=r_{v}(Q)$ for $P,Q \in \G$\\
		$\mathrm{ord} \ T$ & the order of a torsion point $T \in \G$ \\
		$\mathrm{ord}_{v} P$ &  the order of a point $P \bmod v$ \\
		$l^{k} \parallel n$	& means that $l^{k}$ exactly divides $n$, i.e. $l^{k} \mid  n$		and	$l^{k+1} \nmid n$, \\ & where  $l$ is a prime number, $k$ a positive integer, \\ & and $n$ a natural number.				
	\end{tabular}\\
	\\
	Our main technical assumption imposed on the family $r_{v} \colon \G \to \G_{v}$ is the following.  
	\begin{enumerate}[label=A{\arabic*}]
		\item \label{o posylaniu} Let $l$ be a prime number. Let $(k_{1}, \ldots, k_{m})$ be a sequence of nonnegative integers. If $P_{1}, \ldots, P_{m} \in \G$ are linearly independent over $\Z$ then there is a  set of primes $v$ in $\mathbb{K}$ of positive density such that 
		\[l^{k_{i}} \parallel \mathrm{ord}_{v} P_{i} \text{ if } k_{i}>0,\]
		and
		\[l \nmid \mathrm{ord}_{v} P_{i} \text{ if } k_{i}=0.\]
	\end{enumerate}
	Examples of groups known to satisfy the above axioms are the following (see e.g. \cite{Bar2}). 
	\begin{itemize}	
		\item $A(\mathbb{K})$, Mordell-Weil groups of abelian varieties over number fields $\mathbb{K}$ with $\mathrm{End}_{\bar{\mathbb{K}}} (A) = \mathbb{Z}$,
		\item $K_{2n+1}(\mathbb{K})$, $n>0$,  odd algebraic $K$-theory groups of number fields,
		\item finitely generated subgroups of the multiplicative groups of number fields.
	\end{itemize}
	
	The main results of this paper are the following two theorems.
	\begin{thm}\label{thm for 2}
		Let $P, Q\in \G$ be points of infinite order. The following are equivalent:
		\begin{itemize}
			\item For almost every $v$ there exist coprime integers $x,y$ and a point $T \in \G_{\mathrm{tors}}$ such that
			\begin{equation}\label{zal dla 2}
				x^2 P + y^2 Q  = T \mod v.
			\end{equation}
			\item
			There exist coprime integers $x,y$ and a point $T \in \G_{\mathrm{tors}}$ such that
			\begin{equation}\label{teza dla 2}
				x^2 P + y^2 Q  = T.
			\end{equation}
		\end{itemize}
	\end{thm}	
	
	\begin{thm}\label{thm for 3} Let $P, Q, R\in \G$ be points of infinite order. The following are equivalent:
		
		\begin{itemize}
			\item For almost every $v$ there exist integers $x,y,z$ with $\gcd(x,y,z)=1$ and a point $T \in \G_{\mathrm{tors}}$ such that
			\begin{equation}\label{zal}
				x^2 P + y^2 Q + z^2 R = T \mod v.
			\end{equation}
			\item
			There exist integers $x,y,z$ with $\gcd(x,y,z)=1$ and a point $T \in \G_{\mathrm{tors}}$ such that
			\begin{equation}\label{teza}
				x^2 P + y^2 Q + z^2 R = T.
			\end{equation}
		\end{itemize}
	\end{thm}

	Theorems \ref{thm for 2} and \ref{thm for 3} cannot be extended directly to quadratic forms of higher ranks, even if we 
assume that they represent zero for all $v$. Indeed, we have the following.	
\begin{prop}\label{higher}
	Fix $n \ge 4$. Let $P \in \G$ be a point of infinite order. Define the form 
	\[	F(x_{1}, \ldots , x_{n})=( 2 x_{1}^2  + x_{2}^2  +x_{3}^2   +x_{4}^2  + \ldots + x_{n}^2)P.\]
	Then  $F(x_{1}, \ldots , x_{n})\in \G_{\mathrm{tors}}$ if and  only if $(x_{1}, \ldots , x_{n})=(0,\ldots, 0)$. However,
	for every $v$  there exist integers $x_{1}, \ldots , x_{n}$ with $\gcd(x_{1}, \ldots , x_{n})=1$ such that 
	\begin{equation}\label{gt 4 mod v}
		F(x_{1}, \ldots , x_{n}) = 0 \mod v.
	\end{equation}
\end{prop}

\begin{rem}
Our choice of the axiomatic setup for $G$ was dictated by its simplicity to apply to Mordell-Weil or $K$-theory groups. However, we could have replaced the set of primes in number fields as the indexing set by more general, but rather cumbersome setup of infinite sets for which the notion of density is defined. In particular, this would cover the case $\G=\Z$ for which Theorems \ref{thm for 2} and \ref{thm for 3} become Lemma \ref{2and3}, hence justifying our claim that they are its generalization.

Let us demonstrate how Theorem \ref{thm for 2} becomes Lemma \ref{2and3}\emph{\ref{P1}}. First, using Hensel's lifting lemma and Proposition \ref{lifting}, we can replace the local condition in Lemma \ref{2and3}\emph{\ref{P1}} by only apparently stronger  \textit{modulo all natural numbers not divisible by elements of a finite set of prime numbers}. The setup for which Theorem \ref{thm for 2} becomes the reformulated Lemma \ref{2and3}\emph{\ref{P1}} is as follows.
Let $S$ be a finite set of prime numbers,  $\mathcal{V}$ be the set of natural numbers not divisible by the elements in $S$, and for every $v \in \mathcal{V}$ the homomorphism $r_{v} \colon \Z \to \Z/v$ be the usual reduction modulo $v$. Analyzing the proof of Theorem  \ref{thm for 2}, we see that in fact a weaker version of Assumption \ref{o posylaniu} is needed; it has to be satisfied only for all but finitely many prime numbers $l$. This holds  
trivially. Indeed, $\G$ is cyclic, so any set of independent points in $G$ consists of just one point. Let $P\in \Z \setminus \left\lbrace 0 \right\rbrace $. Fix a prime number $l \notin S$, and a nonnegative integer $k$. Let $l^e$ be the maximal power of $l$ dividing $P$. For any natural number $m$ coprime to $l$, to $P$, and to all elements of $S$, define $v=l^{k+e}m \in \mathcal{V}$. The set of such $v$'s has positive density, and we have $\mathrm{ord}_{v} P=l^{k}m$, which means that  $l^k \parallel \mathrm{ord}_{v} P$ if $k$ is positive, and $l \nmid \mathrm{ord}_{v} P$ if $k=0$.   

Taking $S=\emptyset$, we obtain a setup for which Theorem  \ref{thm for 3} is equivalent to Lemma \ref{2and3}\emph{\ref{P2}}; the reasoning is similar.
\end{rem}

\begin{rem}
In Theorems \ref{thm for 2} and \ref{thm for 3}, the phrase \textit{for almost every $v$} means \textit{for all but a set of density zero}. However, our proofs remain valid, if we impose the following three changes simultaneously:
	\begin{itemize}
		\item $r_{v} \colon \G \to \G_{v}$ is a family of group homomorphisms indexed by an infinite set $\mathcal{V}$,
		\item In Assumption \ref{o posylaniu}, the phrase \textit{there is a  set of primes $v$ in $\mathbb{K}$ of positive density} is replaced by \textit{for infinitely many indices  $v \in \mathcal{V}$},
		\item In Theorems \ref{thm for 2} and \ref{thm for 3}, the phrase \textit{for almost every $v$} means \textit{for all but finitely many $v \in \mathcal{V}$}.
	\end{itemize}    
\end{rem}
	
	\begin{rem}
	Considering Theorems \ref{thm for 2} and \ref{thm for 3}, one could wonder if the following principle holds: if a fixed torsion point is represented locally, then it is represented globally. To be precise,
	let $F$ be a quadratic form with coefficients in $\G$, of rank equal to $2$ or $3$. Let $T \in \G_{\mathrm{tors}}$. Does the assumption
	\begin{center}
		\textit{$F$ represents $T$ modulo almost every $v$}	
	\end{center}
	implies
	\begin{center}
		\textit{$F$ represents $T$?}	
	\end{center}
	This is not the case in general, as we demonstrate in  
	Example \ref{example} at the end of the paper.    
	\end{rem}

	\begin{rem}
	  In contrast to classic quadratic forms, i.e., over fields, the quadratic forms over groups in general cannot be diagonalized. This means, that our results does not seem to determine whether the local-global principle holds for nondiagonal forms over $\G$, such as 
	\[x^2 P + xyQ + y^2 R.\]    
	We leave this an open question, since we are neither able to prove such a principle, nor provide a counterexample, even though we could in fact use a much stronger version of Assumption \ref{o posylaniu}, where various prime numbers $l$ can be dealt with simultaneously; see Theorems 2.1.,  A.1., and A.3. in \cite{BanBar}.   
	\end{rem}

	The tools we use in the proofs are (sometimes very delicate) applications of Assumption \ref{o posylaniu}, Lemma \ref{2and3} itself, and variations on the following elementary fact. Let $P$ and $Q$ be elements of finite order in an abelian group, and $l$ be a prime number. If
	\[l^k \parallel \mathrm{ord} \ P \, \text{ and } l^k \nmid \mathrm{ord} \ Q \]
	then
	\[l^k \parallel \mathrm{ord} \ (P +Q).\]

	\begin{proof}[Proof of Lemma \ref{2and3}] 
		By Proposition 1.4., p. 105 in \cite{Neu}, an integral quadratic form  represents $0$ in $\mathbb{Q}_{l}$ if and only if it represents $0$ modulo $l^{k}$, for every positive integer $k$. Thus, \emph{\ref{P2}} is equivalent to  Corollary 3 on p. 43 in \cite{S}. \emph{\ref{P1}} follows from the theorem formulated in Remark 1) on p. 43 in op. cit., and Hensel's lifting lemma. 
	\end{proof}

	\begin{proof}[Proof of Theorem \ref{thm for 2}]
		Suppose that $P$ and $Q$ are linearly independent. Fix a prime number $l$, coprime to $\mathrm{exp}(\G_{\mathrm{tors}})$. By Assumption \ref{o posylaniu}, there is a positive density set of primes $v$ such that 
		\[
		\begin{array}{ccc}
			l^{2} \parallel \mathrm{ord}_{v}P & \mathrm{and} & l^{3} \parallel \mathrm{ord}_{v} Q.
		\end{array}
		\]
		If both $x$ and $y$ are coprime to $l$, then
		\[
		\begin{array}{ccc}
			l^{2} \parallel \mathrm{ord}_{v}(x^2 P) & \mathrm{and} & l^{3} \parallel \mathrm{ord}_{v} (y^2 Q),
		\end{array}
		\]
		so
		\[l^{3} \parallel \mathrm{ord}_{v} (x^2 P+y^2 Q).\]
		If $l \mid x$ and $l \nmid y$, then
		\[
		\begin{array}{ccc}
			l \nmid \mathrm{ord}_{v}(x^2 P) & \mathrm{and} & l^{3} \parallel \mathrm{ord}_{v} (y^2 Q),
		\end{array}
		\]
		so
		\[l^{3} \parallel \mathrm{ord}_{v} (x^2 P+y^2 Q).\]
		If $l \nmid x$ and $l \mid y$, then
		\[
		\begin{array}{ccc}
			l^{2} \parallel \mathrm{ord}_{v}(x^2 P) & \mathrm{and} & l^{2} \nmid \mathrm{ord}_{v} (y^2 Q),
		\end{array}
		\]
		so
		\[l^{2} \parallel \mathrm{ord}_{v} (x^2 P+y^2 Q).\]
		In all cases we get a contradiction to \eqref{zal dla 2}. Hence, $P$ and $Q$  are linearly dependent, i.e., 
		\begin{equation}\label{aPbQ0}
			aP+bQ=0 
		\end{equation}
		for some nonzero integers $a$ and $b$. Multiplying \eqref{zal dla 2} by $b$, and subtracting \eqref{aPbQ0} multiplied by $y^2$, we get
		\begin{equation}\label{tylko P}
			(b x^2 - a y^2)P = b T \mod v.
		\end{equation}
		Fix an arbitrary prime number $l$, coprime to $\mathrm{exp}(\G_{\mathrm{tors}})$. By Assumption \ref{o posylaniu}, there is a positive density set of primes $v$ such that 
		\[l \mid \mathrm{ord}_{v}P.\]
		Therefore, according to \eqref{tylko P}, there are coprime integers $x,y$ such that
		\[l \mid (b x^2 - a y^2).\]
		Thus, by Lemma \ref{2and3} \emph{\ref{P1}}, the form $b x^2 - a y^2$ represents $0$, i.e., there are coprime  integers $x,y$ such that
		\begin{equation}\label{bx2ay20}
			b x^2 = a y^2.
		\end{equation}
		Multiplying \eqref{aPbQ0} by $y^2$, and applying \eqref{bx2ay20} we get
		\[b(x^2 P + y^2 Q)  = 0.\]
		This asserts \eqref{teza dla 2} with $T$ being a torsion point of order dividing $b$.

	\end{proof}

	\begin{proof}[Proof of Theorem \ref{thm for 3}]
		Suppose that $P,Q$ and $R$ are linearly independent. This means that the points 
		\[P+Q, Q, Q+R\] 
		are independent. Let $2^{e}$ be the maximal power of $2$ dividing $\mathrm{exp}(\G_{\mathrm{tors}})$. By Assumption \ref{o posylaniu}, there is a positive density set of primes $v$ such that \begin{equation}\label{ord3 1}
			\begin{array}{ccccc}
				2^{1+e} \parallel \mathrm{ord}_{v} (P+Q) & , & 2^{3+e} \parallel \mathrm{ord}_{v} Q & ,  & 2^{4+e} \parallel \mathrm{ord}_{v} (Q+R). 
			\end{array}
		\end{equation}
		This gives
		\begin{equation}\label{ord3 2}
			\begin{array}{ccccc}
				2^{3+e} \parallel \mathrm{ord}_{v} P & , & 2^{3+e} \parallel \mathrm{ord}_{v} Q & ,  & 2^{4+e} \parallel \mathrm{ord}_{v} R. 
			\end{array}
		\end{equation}
		Now, by \eqref{zal}, we have $2 \mid z$. Thus,
		\begin{equation}\label{ord3 4}
			\begin{array}{cccc}
				\mathrm{either} & 2^{1+e} \nmid \mathrm{ord}_{v} (z^{2} R)& \mathrm{or} & 2^{2+e} \parallel \mathrm{ord}_{v} (z^2 R),
			\end{array}
		\end{equation}
		so, respectively
		\begin{equation}\label{ord3 3}
			\begin{array}{cccc}
				\mathrm{either} & 2^{1+e} \nmid \mathrm{ord}_{v} (x^{2} P + y^2 Q)& \mathrm{or} & 2^{2+e} \parallel \mathrm{ord}_{v} (x^2 P + y^2 Q).
			\end{array}
		\end{equation}
		Since $\gcd(x,y,z)=1$, we have to consider two cases: when both $x$ and $y$ are odd, and when exactly one of $x$ and $y$ is even.  
		
		If both $x$ and $y$ are odd, then $8 \mid (x^2-y^2)$, so we get by \eqref{ord3 1} and \eqref{ord3 2} that 
		\[ \begin{array}{ccc}
			2^{1+e} \parallel \mathrm{ord}_{v} (y^2 (P+Q)) & \mathrm{and} & 2^{1+e} \nmid \mathrm{ord}_{v} ((x^2-y^2) P).
		\end{array}\] 
		Since $x^2 P + y^2 Q = y^2 (P+Q) + (x^2-y^2) P$, we have 
		\[2^{1+e} \parallel \mathrm{ord}_{v} (x^2 P + y^2 Q),\] but that contradicts \eqref{ord3 3}.   
		
		If $x$ is odd and $y$ is even, then we get by \eqref{ord3 2} that
		\[\begin{array}{ccc}
			2^{3+e} \parallel \mathrm{ord}_{v} (x^2 P) & \mathrm{and} & 2^{3+e} \nmid \mathrm{ord}_{v} (y^2 Q), 
		\end{array}\]
		but that together with \eqref{ord3 4} contradicts \eqref{zal}. If $y$ is odd and $x$ is even, then we argue in the same manner.
		
		In all cases, $P,Q$ and $R$ are linearly dependent.\\
		
		First, we analyze the case in which among the pairs $\left\lbrace P, Q \right\rbrace $, $\left\lbrace P, R \right\rbrace $ and  $\left\lbrace Q, R \right\rbrace $, one is a pair of linearly dependent points. Without loss of generality, we can assume that this pair is $\left\lbrace P, Q \right\rbrace $, i.e., there are nonzero rational integers $s,t$ such that
		\begin{equation}\label{sPtQ}
			s P = t Q .
		\end{equation}
		
		If $-s/t$ is the square of a rational number, then we are done. Indeed, write
		\[-s/t=(p/q)^2\] with coprime $p,q$. Together with \eqref{sPtQ}, it gives
		\[t (p^2 P+ q^2 Q)=0.\]
		This asserts \eqref{teza} with $(x,y,z)=(p,q,0)$, and $T$ being  a torsion point of order dividing $t$.
		
		So suppose that $-s/t$ is not a square. Lemma \ref{2and3} \emph{\ref{P1}} asserts that there are infinitely many prime numbers $l$ such that $-s/t$ is a quadratic nonresidue modulo $l$. Fix one such $l$, coprime to both $s,t$, and to $\mathrm{exp}(\G_{\mathrm{tors}})$. 
		Multiplying \eqref{zal} by $s$, and using \eqref{sPtQ} we get  
		\begin{equation}\label{zal s}
			(t x^2  + s y^2 ) Q +  s z^2 R = s T \mod v .
		\end{equation}		
		Suppose that $Q$ and $R$ are linearly independent. By Assumption \ref{o posylaniu},  there is a positive density set of primes $v$ such that 
		\[\begin{array}{ccc}l^2 \parallel \mathrm{ord}_{v} Q & \mathrm{and} & l^3 \parallel \mathrm{ord}_{v} R.\end{array}\]
		This, together with \eqref{zal s}, gives us the following chain of implications. First, we get that
		\[l \mid z.\] This means that 
		\[\begin{array}{cccc}
			\mathrm{either}& l \parallel \mathrm{ord}_{v} (s z^2 R)& \mathrm{or}& l \nmid \mathrm{ord}_{v} (s z^2 R),
		\end{array}\]
		so 
		\[ l^2 \nmid \mathrm{ord}_{v} ((t x^2  + s y^2 ) Q).
		\]
		Thus
		\[l \mid (t x^2  + s y^2 ). \] 
		However, $\gcd(x,y,z)=1$. Hence, both $x$ and $y$ are coprime to $l$, so we have
		\[(x/y)^2=-s/t \mod l ,\]
		contrary to our assumption that 
		$-s/t$ is nonresidue. Thus $Q$ and $R$ are linearly dependent, i.e.,
		\begin{equation}\label{uQwR}
			uQ=wR
		\end{equation}
		for some nonzero rational integers $u,w$. Multiplying \eqref{zal s} by $w$, and using \eqref{uQwR} multiplied by $s z^2$, we get  
		\begin{equation}\label{zal s w}
			( w t x^2  + w s y^2 +  u s z^2) Q  = w s T \mod v .
		\end{equation}
		Denote by $e$ the maximum of the exponents occurring in the prime decomposition of $\mathrm{exp}(\G_{\mathrm{tors}})$.
		Fix an arbitrary prime number $l$, and a positive integer $k$. By Assumption \ref{o posylaniu}, there is a positive density set of primes $v$ such that \[l^{k+e} \parallel \mathrm{ord}_{v} Q .\]
		By \eqref{zal s w}, we get that
		\[l^{k} \mid ( w t x^2  + w s y^2 +  u s z^2) .\]
		Thus, by Lemma \ref{2and3} \emph{\ref{P2}}, the quadratic form $( w t x^2  + w s y^2 +  u s z^2)$ represents 0, i.e., there are integers $x,y,z$, with $\gcd(x,y,z)=1$ such that
		\[
		w t x^2  + w s y^2 +  u s z^2 = 0.\]
		This, together with \eqref{sPtQ} and \eqref{uQwR}, gives 
		\[0=(w t x^2  + w s y^2 +  u s z^2) Q = ws(x^2 P + y^2 Q + z^2 R).\]
		This asserts \eqref{teza} with $T$ being a torsion point of order dividing $ws$.\\

		Now, we turn to the case in which among the pairs $\left\lbrace P, Q \right\rbrace $, $\left\lbrace P, R \right\rbrace $ and $\left\lbrace Q, R \right\rbrace $, none is a pair of linearly dependent points. Write
		\begin{equation}\label{lin dep no pair}
			aP+bQ+cR=0
		\end{equation}
		with $a,b,c \ne 0$. Multiplying \eqref{zal} by  $a$ and \eqref{lin dep no pair} by $x^2$ and subtracting the results, we get
		\begin{equation}\label{lin dep QR}
			(ay^2-bx^2)Q+(az^2-cx^2)R=a T \mod v.
		\end{equation}
		Let $l$ be a prime number,  coprime to the numbers $a,b,c,\mathrm{exp}(\G_{\mathrm{tors}})$. By Assumption \ref{o posylaniu}, there is a positive density set of primes $v$ such that 
		\begin{equation}\label{PQorders}
			\begin{array}{ccc}
				l^{2} \parallel \mathrm{ord}_{v}Q & \mathrm{and} & l \parallel \mathrm{ord}_{v} R.
			\end{array}
		\end{equation}
		By \eqref{lin dep QR}, we get
		\begin{equation}\label{abxy}
			l \mid (ay^2-bx^2) .
		\end{equation} 
		Thus, if $l$ divides either $x$ or $y$, then it divides them both. Assume that this is the case. However, this implies by \eqref{PQorders} and \eqref{lin dep QR}, that 
		\[l \nmid \mathrm{ord}_{v}(az^2R).\]
		Hence, by \eqref{PQorders},  we have $l\mid z$. This contradicts our assumption that $\gcd(x,y,z)=1$. Thus, $l$ divides neither $x$ nor $y$, so \eqref{abxy} means that $b/a$ is a square modulo $l$. 
		
		The only restriction on $l$ in the above argument is that it cannot divide any of the numbers $a,b,c$, and $\mathrm{exp}(\G_{\mathrm{tors}})$. Hence, it  is valid for all but finitely many prime numbers $l$, so Lemma \ref{2and3} \emph{\ref{P1}}  asserts that $b/a$ is the square of a rational number. Similarly, we prove the same for $c/a$. 
		
		Thus, we can set $b/a=(t/s)^2$ with coprime $t,s$, and $c/a=(u/w)^2$  with coprime $u,w$, so by \eqref{lin dep no pair} we get
		\[a(sw)^2P+a(tw)^2Q+a(su)^2R=0.\] 
		This asserts \eqref{teza} with $x=sw/\gcd(w,s),y=tw/\gcd(w,s),z=su/\gcd(w,s)$, and  $T$ being a torsion point of order dividing $a (\gcd(w,s))^2$.\\
		
	\end{proof}

	\begin{proof}[Proof of Proposition \ref{higher}] The form \[ 2 x_{1}^2  + x_{2}^2  +x_{3}^2   +x_{4}^2  + \ldots + x_{n}^2\] is positive, and point $P$ is of infinite order, so \[( 2 x_{1}^2  + x_{2}^2  +x_{3}^2   +x_{4}^2  + \ldots + x_{n}^2)P = 0\] if and  only if $(x_{1}, \ldots , x_{n})=(0,\ldots, 0)$.
		
		Let $k$ be the order of $P\bmod v$. This is an easy  corollary of the Gauss theorem  on sums of three squares, that $2k$ is of the form $2 a^2  + b^2  +c^2 +1$, where $a,b$ and $c$ are integers (cf. \cite{Sie}, page 395, exercise 4). Thus, in order to get \eqref{gt 4 mod v}, we  set $x_{1}=a, x_{2}=b, x_{3}=c, x_{4}=1$, and $x_{i}=0$ for $i > 4$.
	\end{proof}
	

	\begin{example}\label{example}
		Suppose that there is a nontorsion point $P\in \G$, and points $T_{1}, T_{2} \in \G_{\mathrm{tors}}$ such that $\mathrm{ord} \ T_{1} = 5$ and $\mathrm{ord} \ T_{2} = 2$. Consider the form
		\[F(x,y)=x^2 (P+T_{1}) - y^2 (P+T_{2}).\]
		The unique torsion point represented by $F$, i.e., equal to $F(x,y)$ for coprime $x,y$, is 
		\[F(1,1)=T_{1}-T_{2}.\]
		However, modulo every $v$, $F$ represents the point 
		\[4T_{1}-T_{2},\]
		which is different from $F(1,1)$, provided the restriction of $r_{v}$ to $\G_{\mathrm{tors}}$ is injective.
		Indeed, define
		\[n=2 \ \mathrm{ord}_{v} P,\]
		and factorize
		\[n=2^e m,\]
		where $m$ is an odd number.
		
		Let $k$ be a solution to the following system of congruences
		\[\left\{ \begin{array}{l}
			3k+2 \equiv 0 \bmod 2^e \\
			4k+3 \equiv 0 \bmod m .
		\end{array} \right.
		\]
		Define
		\[\begin{array}{l}
			x=10k+7, \\
			y=2k+1 .
		\end{array} 
		\]
		We have
		\[\gcd(x,y)=\gcd(x-5y,y)=\gcd(2,2k+1)=1,\]
		and 
		\[x^2-y^2=(x+y)(x-y)=(12k+8)(8k+6)=8(3k+2)(4k+3) \equiv 0 \bmod n.\]
		Thus, we get the claimed representation
		\[F(x,y)=(x^2-y^2) P+ x^2  T_{1} - y^2  T_{2} \equiv 4 T_{1} -   T_{2} \bmod v. \]
		
		We conclude by providing  concrete examples of $G, P, T_{1}$ and $T_{2}$. Consider the elliptic curve $E/\Q$ 
		\[y^2+xy=x^3-454955x+118072977.\] 
		According to the $L$-functions and modular forms database (LMFDB), this curve (LMFDB label 6270.r2) has the Mordell-Weil group $E(\Q)$ isomorphic to $\Z \oplus \Z/10$. Let $G_{1}$ and $G_{2}$ be its infinite order generator, and torsion generator, respectively. Take $\G=E(\Q)$, $P=G_{1}$, $T_{1}=2 G_{2}$, $T_{2}=5 G_{2}$. 
		
		Recall that if $E/\mathbb{K}$ is an elliptic curve over a number field, then for all but finitely many primes of good reduction $v$, the restriction of the corresponding reduction map to $E(\mathbb{K})_{\mathrm{tors}}$ is injective.	
	\end{example}

\begin{prop}\label{lifting}
	Let $f$ be an integral quadratic form, and $u,v>1$ be coprime natural numbers. Suppose there exist sequences of integers  $x_{1},\ldots,x_{n}$ and $y_{1},\ldots,y_{n}$  such that $\gcd(x_{1},\ldots,x_{n})=\gcd(y_{1},\ldots,y_{n})=1$ and 
	\begin{equation}\label{ChRT for u and v}
		\left\lbrace \begin{array}{l}
			f(x_{1},\ldots,x_{n})\equiv 0 \bmod u\\
			f(y_{1},\ldots,y_{n})\equiv 0 \bmod v.
		\end{array}\right.  
	\end{equation}
	Then there exists a sequence of  integers $z_{1},\ldots,z_{n}$ such that $\gcd(z_{1},\ldots,z_{n})=1$ and
	\begin{equation}\label{ChRT for uv}
			f(z_{1},\ldots,z_{n})\equiv 0 \bmod uv .  
	\end{equation}
\end{prop}
\begin{proof}
	Since $u$ and $v$ are coprime, there exist integers $s$ and $t$ such that
	\begin{equation}\label{xeucl for u v}
		su+tv=1.
	\end{equation}
	For every $i \in \left\lbrace 1, \ldots, n\right\rbrace $ define
	\begin{equation}\label{def of zeta}
		\zeta_{i}=suy_{i}+tvx_{i}.
	\end{equation}
By \eqref{ChRT for u and v} we get
\begin{equation}\label{ChRT zeta for uv}
	f(\zeta_{1},\ldots,\zeta_{n})\equiv 0 \bmod uv .  
\end{equation}
	Let $p$ be a prime number dividing $\gcd(\zeta_{1},\ldots,\zeta_{n})$. Suppose that $p \mid u$. By \eqref{def of zeta} and \eqref{xeucl for u v} we get that $p \mid x_{i}$ for every $i \in \left\lbrace 1, \ldots, n\right\rbrace $, contrary to the assumption that $\gcd(x_{1},\ldots,x_{n})=1$. Thus $\gcd(\zeta_{1},\ldots,\zeta_{n})$ and $u$ are coprime. Analogously, we get the same for $v$. Hence $\gcd(\zeta_{1},\ldots,\zeta_{n})$ is invertible modulo $uv$. Dividing \eqref{ChRT zeta for uv} by the square of $\gcd(\zeta_{1},\ldots,\zeta_{n})$ and putting
	\[z_{i}=\zeta_{i}/\gcd(\zeta_{1},\ldots,\zeta_{n})\]
	we get \eqref{ChRT for uv}.   
\end{proof}

	\section*{Acknowledgements}
	We are grateful to Grzegorz Banaszak for the discussions.


	\bibliographystyle{plain}

\end{document}